\documentclass{amsart}

\title[$V$ may not be a class-forcing extension of \HOD]{The set-theoretic universe $V$ is not necessarily a class-forcing extension of $\HOD$}

\author[Hamkins]{Joel David Hamkins}
 \address[Joel David Hamkins]
         {Mathematics, Philosophy, Computer Science, The Graduate Center of The City University of New York,
         365 Fifth Avenue, New York, NY 10016 \&
         Mathematics, College of Staten Island of CUNY, Staten Island, NY 10314 USA}
 \email{jhamkins@gc.cuny.edu}
 \urladdr{http://jdh.hamkins.org}

\author[Reitz]{Jonas Reitz}
 \address[Jonas Reitz]{New York City College of Technology of The City University of New York,
                    Mathematics, 300 Jay Street, Brooklyn, NY 11201 USA}
 \email{jonasreitz@gmail.com}

\thanks{The research of the second author has been supported in part by PSC-CUNY grant 69095-00-47. This article grew out of an exchange of the authors on MathOverflow~\cite{Reitz2013.MO118477:Is-every-class-that-does-not-add-sets-necessarily-added-by-forcing, Hamkins2013.MO118525:Is-every-class-that-does-not-add-sets-necessarily-added-by-forcing}, which led to the main technique. Questions and commentary on this article can be made at {\smaller\url{http://jdh.hamkins.org/the-universe-need-not-be-a-class-forcing-extension-of-hod}}.}

\usepackage{relsize}
\usepackage[backend=bibtex,style=alphabetic,maxbibnames=15,maxcitenames=6,dateabbrev=false]{biblatex}
\renewcommand{\UrlFont}{\smaller} 
\renewbibmacro{in:}{\ifentrytype{article}{}{\printtext{\bibstring{in}\intitlepunct}}} 
\DeclareFieldFormat{url}{{\relsize{-2}\UrlFont\url{#1}}} 
\DeclareFieldFormat{urldate}{
  (version \thefield{urlday}\addspace%
  \mkbibmonth{\thefield{urlmonth}}\addspace%
  \thefield{urlyear}\isdot)}
\DeclareFieldFormat{eprint:arxiv}{
  \ifhyperref
    {\href{http://arxiv.org/abs/#1}{%
        arXiv\addcolon\nolinkurl{#1}}\iffieldundef{eprintclass}{}{\UrlFont{\mkbibbrackets{\thefield{eprintclass}}}}}
    {arXiv\addcolon\nolinkurl{#1}\iffieldundef{eprintclass}{}{\UrlFont{\mkbibbrackets{\thefield{eprintclass}}}}}}
\bibliography{HamkinsBiblio,MathBiblio,LogicBiblio,WebPosts}
\newbox\gnBoxA
\newdimen\gnCornerHgt
\setbox\gnBoxA=\hbox{\tiny$\ulcorner$}
\global\gnCornerHgt=\ht\gnBoxA
\newdimen\gnArgHgt
\def\gcode #1{%
\setbox\gnBoxA=\hbox{$#1$}%
\gnArgHgt=\ht\gnBoxA%
\ifnum     \gnArgHgt<\gnCornerHgt \gnArgHgt=0pt%
\else \advance \gnArgHgt by -\gnCornerHgt%
\fi \raise\gnArgHgt\hbox{\tiny$\ulcorner$} \box\gnBoxA %
\raise\gnArgHgt\hbox{\tiny$\urcorner$}}
\usepackage[utf8]{inputenc}
\usepackage{amssymb,bbm}
\usepackage[usenames,dvipsnames,svgnames,table]{xcolor} 
\usepackage[hidelinks,colorlinks=true,linkcolor=red!35!black,citecolor=green!35!black, urlcolor=blue!35!black]{hyperref}
\usepackage[shortlabels]{enumitem} 
\usepackage{bbm}
\newtheorem{theorem}{Theorem}

\newtheorem{question}[theorem]{Question}

\newcommand{\Vopenka}{Vop\v{e}nka}

\newcommand{\Godel}{G\"odel}

\newcommand{\Lowenheim}{L\"owenheim}

\newcommand{\B}{{\mathbb B}}

\renewcommand{\P}{{\mathbb P}}

\newcommand{\Q}{{\mathbb Q}}

\newcommand{\of}{\subseteq}

\newcommand{\set}[1]{\{\,{#1}\,\}}

\newcommand{\satisfies}{\models}
\newcommand{\forces}{\Vdash}

\newcommand{\Mantle}{{\mathord{\rm M}}}

\newcommand{\concat}{\mathbin{{}^\smallfrown}}

\newcommand{\intersect}{\cap}
\newcommand{\Intersect}{\bigcap}

\def\<#1>{\langle#1\rangle}

\newcommand{\val}{\mathop{\rm val}\nolimits}

\newcommand{\Ord}{\mathop{{\rm Ord}}}

\newcommand{\ZFC}{{\rm ZFC}}
\newcommand{\ZF}{{\rm ZF}}

\newcommand{\KM}{{\rm KM}}

\newcommand{\GBC}{{\rm GBC}}

\newcommand{\GCH}{{\rm GCH}}

\newcommand{\HOD}{{\rm HOD}}
\newcommand{\OD}{{\rm OD}}

\newcommand{\ETR}{{\rm ETR}}
\newcommand{\Tr}{\mathop{\rm Tr}}

\newcommand{\Add}{\mathop{\rm Add}}
\newcommand{\elesub}{\prec}

\newtheorem*{maintheorem*}{Main Theorem} 
\newtheorem*{mainquestion*}{Main Question} 

\begin{document}

\begin{abstract}
In light of the celebrated theorem of \Vopenka~\cite{Vopenka&Hajek1972:TheTheoryOfSemisets}, proving in \ZFC\ that every set is generic over \HOD, it is natural to inquire whether the set-theoretic universe $V$ must be a class-forcing extension of \HOD\ by some possibly proper-class forcing notion in \HOD. We show, negatively, that if \ZFC\ is consistent, then there is a model of \ZFC\ that is not a class-forcing extension of its \HOD\ for any class forcing notion definable in \HOD\ and with definable forcing relations there (allowing parameters). Meanwhile, S.~Friedman~\cite{Friedman2012:TheStableCore} showed, positively, that if one augments $\HOD$ with a certain \ZFC-amenable class $A$, definable in $V$, then the set-theoretic universe $V$ is a class-forcing extension of the expanded structure $\<\HOD,\in,A>$. Our result shows that this augmentation process can be necessary. The same example shows that $V$ is not necessarily a class-forcing extension of the mantle, and the method provides counterexamples to the intermediate model property, namely, a class-forcing extension $V\of W\of V[G]$ with an intermediate transitive inner model $W$ that is neither a class-forcing extension of $V$ nor a ground model of $V[G]$ by any definable class forcing notion with definable forcing relations.
\end{abstract}

\maketitle

\section{Introduction}

In 1972, \Vopenka\ proved the following celebrated result.

\begin{theorem}[\Vopenka~\cite{Vopenka&Hajek1972:TheTheoryOfSemisets}]\label{Theorem.Vopenka-theorem-every-set-generic-over-HOD}
 If $V=L[A]$ where $A$ is a set of ordinals, then $V$ is a forcing extension of the inner model $\HOD$.
\end{theorem}

The result is now standard, appearing in Jech~\cite[p. 249]{Jech:SetTheory3rdEdition} and elsewhere, and the usual proof establishes a stronger result, stated in \ZFC\ simply as the assertion: every set is generic over \HOD. In other words, for every set $a$ there is a forcing notion $\B\in\HOD$ and a $\HOD$-generic filter $G\of\B$ for which $a\in\HOD[G]\of V$. The full set-theoretic universe $V$ is therefore the union of all these various set-forcing generic extensions $\HOD[G]$.

It is natural to wonder whether these various forcing extensions $\HOD[G]$ can be unified or amalgamated to realize $V$ as a single class-forcing extension of \HOD\ by a possibly proper class forcing notion in \HOD. We expect that it must be a very high proportion of set theorists and set-theory graduate students, who upon first learning of \Vopenka's theorem, immediately ask this question.

\begin{mainquestion*}\label{Question.Main-question}
 Must the set-theoretic universe $V$ be a class-forcing extension of \HOD?
\end{mainquestion*}

We intend the question to be asking more specifically whether the universe $V$ arises as a bona-fide class-forcing extension of \HOD, in the sense that there is a class forcing notion $\P$, possibly a proper class, which is definable in $\HOD$ and which has definable forcing relation $p\forces\varphi(\tau)$ there for any desired first-order formula $\varphi$, such that $V$ arises as a forcing extension $V=\HOD[G]$ for some $\HOD$-generic filter $G\of\P$, not necessarily definable.

In this article, we shall answer the question negatively, by providing a model of \ZFC\ that cannot be realized as such a class-forcing extension of its \HOD.

\begin{maintheorem*}\label{Theorem.Main-theorem}
 If \ZFC\ is consistent, then there is a model of $\ZFC$ which is not a forcing extension of its $\HOD$ by any class forcing notion definable in that \HOD\ and having a definable forcing relation there.
\end{maintheorem*}

Throughout this article, when we say that a class is definable, we mean that it is definable in the first-order language of set theory allowing set parameters.

The main theorem should be placed in contrast to the following result of Sy Friedman, which we discuss in section~\ref{Section.Friedmans-theorem}.

\begin{theorem}[Friedman~\cite{Friedman2012:TheStableCore}]\label{Theorem.Friedmans-theorem}
There is a definable class $A$, which is strongly amenable to \HOD, such that the set-theoretic universe $V$ is a generic extension of $\<\HOD,\in,A>$.
\end{theorem}

This is a postive answer to the main question, if one is willing to augment \HOD\ with a class $A$ that may not be definable in \HOD. Our main theorem shows that in general, this kind of augmentation process is necessary.

In section~\ref{Section.The-mantle}, we extend our analysis to provide a model of \ZFC\ which cannot arise as a class-forcing extension of its mantle, by any class forcing notion definable in the mantle. And in section~\ref{Section.Intermediate-models}, we consider the intermediate model property for class forcing, the question of whether every transitive inner model $W$ that is intermediate $V\of W\of V[G]$ between a ground model $V$ and a class-forcing extension $V[G]$ must itself arise as a class-forcing extension of $V$.

\section{Some set-theoretic background}\label{Section.Background}

An object is definable in a model if it satisfies a property in that model that only it satisfies. In a model of set theory, a set is \emph{ordinal-definable} if it satisfies a property, expressible in the first-order language of set theory with some ordinal parameters, such that only it has that property with those parameters. Although this is a model-theoretic definition, undertaken from an outside-the-universe perspective making reference to a truth predicate (in order to speak of which objects and how many satisfy the given formula), nevertheless it is a remarkable fact that in a model of \ZF\ set theory, we can define the class of ordinal-definable sets internally. Namely, in set theory we define \OD\ as the class of sets $a$ such that for some ordinal $\theta$, the set $a$ is ordinal-definable in the set structure $\<V_\theta,\in>$. In \ZF, we have truth predicates for every set structure, including the structures $\<V_\theta,\in>$, and so this definition can be carried out internally in \ZF, defining a certain class. The remarkable observation is that if $a$ happens actually to be ordinal-definable in the full universe $\<V,\in>$, then it is the unique object satisfying $\varphi(a,\vec\alpha)$ for some formula $\varphi$ and ordinals $\vec\alpha$, and by the reflection theorem there is some $V_\theta$ for which this property is true in $V_\theta$, and so $a$ will be defined by $\varphi(\cdot,\vec\alpha)$ in $V_\theta$, placing $a\in\OD$. Conversely, if $a\in \OD$, then $a$ is defined by some (internal) formula $\varphi(\cdot,\vec\alpha)$ in $V_\theta$, and this is enough to conclude that $a$ is ordinal definable in $V$, because $a$ is the unique object that satisfies $\varphi(a,\vec\alpha)$ in $V_\theta$.\footnote{This argument is more subtle than it may at first appear, since the claim is that this works in every model of \ZFC, including the $\omega$-nonstandard models. The subtle issues are first, that perhaps the finite sequence of ordinals $\vec\alpha$ isn't really finite, but only nonstandard finite; and second, perhaps the formula $\varphi$ itself isn't a standard-length formula, but is a nonstandard formula, which we cannot use to define $a$ in $V$ externally. The first issue can be resolved by using the internal pairing function of the model to replace $\vec\alpha$ with a single ordinal $\alpha=\gcode{\vec\alpha}$ that codes it, thereby reducing to the case of one ordinal parameter. (In fact, one can omit the need for any ordinal parameter at all other than $\theta$, by coding information into $\theta$.) The second issue is resolved simply by taking the \Godel\ code $\gcode{\varphi}$ as an additional parameter. Since this is a natural number, it is an ordinal and therefore can be used as an ordinal parameter. So $a$ is the unique object in $V$ for which $\varphi(a,\vec\alpha)$ holds in $V_\theta$, which is expressible in $V$ as a property of $a$, $\theta$, $\gcode{\vec\alpha}$ and $\gcode{\varphi}$. So it is ordinal-definable.}

Having defined \OD, we define the class \HOD\ of \emph{hereditarily ordinal-definable} sets, which are the sets $a$ whose transitive closure is contained in \OD. So $a$ is ordinal definable, as well as all of its elements and elements-of-elements and so on. It follows that \HOD\ is a transitive class containing all the ordinals, and it is an elementary observation, originally due to \Godel, to prove that \HOD\ is a model of \ZFC. From the fact that the objects of \HOD\ are there in virtue of a definition with ordinal parameters, it follows that in $V$ we may define a global well-order of \HOD.

Several of the arguments will rely on the concept of a satisfaction class or truth predicate. Specifically, a class $\Tr$ is a \emph{satisfaction class} or \emph{truth predicate} for first-order set-theoretic assertions, if it consists of pairs $\<\varphi,\vec a>$ where $\varphi$ is a first-order formula in the language of set theory and $\vec a$ is a valuation of the free-variables of $\varphi$ to actual objects in the set-theoretic universe $V$, such that the Tarskian recursion is satisfied:
\begin{enumerate}[(a)]
  \item $\Tr$ judges the truth of atomic statements correctly:
            \begin{align*}
              \Tr(x=y,\<a,b>) &\quad \text{ if and only if }\quad a=b \\
              \Tr(x\in y,\<a,b>) &\quad \text{ if and only if } \quad a\in b
            \end{align*}
  \item $\Tr$ performs Boolean logic correctly:
            \begin{align*}
              \Tr(\varphi\wedge\psi,\vec a)&\quad \text{ if and only if }\quad \Tr(\varphi,\vec a)\text{ and }\Tr(\psi,\vec a)\\
              \Tr(\neg\varphi,\vec a)\ &\quad \text{ if and only if }\quad \neg\Tr(\varphi,\vec a)
            \end{align*}
  \item $\Tr$ performs quantifier logic correctly:
            $$\Tr(\forall x\, \varphi,\vec a)\quad\text{ if and only if }\quad\forall b\, \Tr(\varphi,b\concat\vec a)$$
\end{enumerate}
In any model $\<M,\in,\Tr>$ equipped with such a truth predicate, an easy induction on formulas shows that $\Tr(\varphi,\vec a)$ if and only if $M\satisfies\varphi[\vec a]$, for any (standard) formula $\varphi$ in the language of set theory. (In $\omega$-nonstandard models, the truth predicate will also involve nonstandard formulas.) Tarski famously proved that no model of set theory can have a definable truth predicate, even with parameters. Please see~\cite{GitmanHamkinsHolySchlichtWilliams:The-exact-strength-of-the-class-forcing-theorem} for further discussion of truth predicates of various kinds in the \Godel-Bernays \GBC\ set theory context.

If $\kappa$ is an inaccessible cardinal, then $V_\kappa$ is a model of \ZFC\ and indeed a model of Kelley-Morse set theory \KM, when equipped with all its subsets as classes, and among these is a truth predicate for the structure $\<V_\kappa,\in>$. Indeed, \KM\ itself implies the existence of a truth predicate for first-order truth, as explained in~\cite{Hamkins.blog2014:KMImpliesCon(ZFC)AndMuchMore}, and we don't actually need the full power of \KM\ for this, since the principle of elementary transfinite recursion \ETR, and indeed, already a small fragment $\ETR_\omega$ suffices, a principle weaker than the principle of clopen determinacy for class games~\cite{GitmanHamkins2016:OpenDeterminacyForClassGames}. But it is not provable in \GBC. So the existence of a truth predicate is a mild extension of \GBC\ in the direction of \ETR\ on the way to \KM.

The article~\cite{GitmanHamkinsHolySchlichtWilliams:The-exact-strength-of-the-class-forcing-theorem} also makes clear exactly what it means to say in \GBC\ that a class forcing notion $\P$ admits a forcing relation for a first-order formula $\varphi$. For example, to admit a forcing relation for the atomic formulas $x\in y$, $x=y$, means to have a class relation $\forces$ for which the instances $p\forces \tau\in\sigma$ and $p\forces \tau=\sigma$ obey the desired recursive properties, much like having a truth predicate. And there are similar recursive requirements for a forcing relation for any given formula $p\forces\varphi(\tau)$. Furthermore, the difficulty of the existence of forcing relations is present entirely in the atomic case, for if a class forcing notion $\P$ admits a forcing relation for the atomic formulas, then it has forcing relations $p\forces\varphi(\vec\tau)$ for any specific first-order formula $\varphi$ (see~\cite[theorem~5]{GitmanHamkinsHolySchlichtWilliams:The-exact-strength-of-the-class-forcing-theorem}). When a class forcing notion admits a forcing relation, then it follows as a consequence that forced statements are true in the forcing extensions and true statements are forced, in the expected manner (see~\cite[section~3]{GitmanHamkinsHolySchlichtWilliams:The-exact-strength-of-the-class-forcing-theorem}). Meanwhile, the main result of~\cite{GitmanHamkinsHolySchlichtWilliams:The-exact-strength-of-the-class-forcing-theorem} shows that the assertion that every class forcing notion $\P$ admits a forcing relation for atomic formulas is equivalent to the principle $\ETR_{\Ord}$, and this in turn implies the existence of various kinds of truth predicates and even iterated truth predicates and furthermore, that every forcing notion $\P$ admits a uniform forcing relation that works with all formulas $\varphi$ simultaneously.

\section{The main theorem}

Let us now prove the main theorem. We shall prove a sequence of variations of it, beginning with an argument which achieves some greater transparency (and the stronger conclusion of a transitive model) by making use of an inaccessible cardinal, a mild extra assumption that will be weakened and ultimately omitted in the later variations.

\begin{theorem}\label{Theorem.Main-theorem-using-inaccessible}
 If there is an inaccessible cardinal, then there is transitive model of \ZFC\ which does not arise as a class-forcing extension of its \HOD\ by any class forcing notion with definable forcing relations in that \HOD.
\end{theorem}

\begin{proof}
Suppose that $\kappa$ is inaccessible, and assume without loss that $V=L$, by moving if necessary to the constructible universe $L$. Consider the model $V_\kappa$, which is the same as $L_\kappa$ and which is a model of \ZFC. Since $\kappa$ is inaccessible, we may augment this model with all its subsets $\<V_\kappa,\in,V_{\kappa+1}>$, which will be a model of the second-order Kelley-Morse \KM\ set theory and therefore also a model of \Godel-Bernays \GBC\ set theory. Let $\Tr=\set{\<\varphi,\vec a>\mid V_\kappa\satisfies\varphi[\vec a]}$ be the satisfaction relation for $V_\kappa$, which is defined by the Tarski recursion in $V$; this is one of the classes in the \KM\ model. Let $T\of\kappa$ be a class of ordinals that codes the class $\Tr$ in some canonical manner, such as by including $\alpha$ if and only if the $\alpha^{th}$ element in the $L$-order is a pair $\<\varphi,\vec a>$ that is in $\Tr$. Note that every initial segment $T\intersect\gamma$ for $\gamma<\kappa$ is an element of $L_\kappa$, since it is the $\gamma$-initial segment of the truth predicate of some $L_\theta\elesub L_\kappa$, which is of course in $L_\kappa$. The set $T$ is a class in our model $\<V_\kappa,\in,V_{\kappa+1}>$, but in light of Tarski's theorem it is not first-order definable in the first-order part of this model $\<V_\kappa,\in>$.

Let $\delta_\alpha=\aleph_{\omega\cdot(\alpha+1)+1}$ for $\alpha<\kappa$; these will constitute a definable sequence of regular cardinals to be used as `coding points,' and sufficiently scattered to avoid any possible interference (this scattering is not strictly necessary, but it makes matters clear). Consider now the Easton-support product forcing $\P=\prod_{\alpha\in T}\Add(\delta_\alpha,1)$, which adds a Cohen subset to each $\delta_\alpha$ for $\alpha\in T$. This is set-forcing in $V$, but definable class forcing in our model $V_\kappa$; it is a very well understood forcing notion, and the standard Easton forcing factor arguments show that it preserves all cardinals and cofinalities of $V$ and that it preserves the inaccessibility of the cardinal $\kappa$. Because the forcing is pre-tame, the forcing relations $p\forces_\P\varphi(\vec \tau)$ for this forcing notion over $V_\kappa$ are definable using class parameter $T$ in $V_\kappa$. Suppose that $G\of\P$ is $V$-generic for this forcing, and in $V[G]$ we may consider the model $\<V_\kappa[G],\in>$, where we have now discarded all the classes and consider it merely as a \ZFC\ model.

We claim first that $\HOD^{V_\kappa[G]}=V_\kappa$. Since $V_\kappa=L_\kappa$, of course we have the inclusion $V_\kappa\of\HOD^{V_\kappa[G]}$. Conversely, notice that the forcing $\P$ is weakly homogeneous in $V$, which means that for any two conditions $p$ and $q$ there is an automorphism of the forcing $\pi$ for which $\pi(p)$ is compatible with $q$---one simply flips bits on the relevant coordinates in order to bring the conditions into compatibility. Since $p\forces\varphi(\check \alpha)$ if and only if $\pi(p)\forces\varphi(\check\alpha)$, it follows that any statement in the forcing language about check-names will be forced either by all conditions or none. Since every ordinal in the extension has a check name, it follows that every ordinal-definable set in $V_\kappa[G]$ is ordinal-definable in $V_\kappa$ from the forcing relation, which is definable from $T$. Since any set in \HOD\ is coded by such a set of ordinals in \HOD, it follows that every set in $\HOD^{V_\kappa[G]}$ is definable from set parameters in the structure $\<V_\kappa,\in,T>$ and in particular, it is an element of $V_\kappa$, since this structure satisfies the separation axiom. So $\HOD^{V_\kappa[G]}=V_\kappa=L_\kappa$.

Next, we notice that $T$ is first-order definable in the extension $V_\kappa[G]$, since it is precisely the collection of ordinals $\alpha$ for which there is in $V_\kappa[G]$ an $L$-generic Cohen subset of $\delta_\alpha$, using the fact that the map $\alpha\mapsto\delta_\alpha$ is definable in a way that is absolute between $V_\kappa$ and $V_\kappa[G]$. The forcing $G\of\P$ explicitly added such a generic set for each $\alpha\in T$ and an easy factor argument shows that when $\alpha\notin T$, there can be no such $L$-generic Cohen subset of $\delta_\alpha$ in $V_\kappa[G]$, since it cannot be added by the tail forcing, which is closed, and it cannot be added by the smaller factors, by the chain condition, since the forcing below $\delta_\alpha$ is too small on account of the scattering. So the class $T$ is first-order definable in $V[G]$. Consequently also, the $V_\kappa$-truth predicate $\Tr$ is definable as well, since this is definable from $T$.

Lastly, we claim that there can be no forcing notion $\Q\of V_\kappa$ having definable forcing relations in $V_\kappa$, such that $V_\kappa[G]$ arises as a forcing extension $V_\kappa[H]$ by some $V_\kappa$-generic filter $H\of\Q$. Suppose toward contradiction that there is. Note that $\Q$ itself, being the domain of its forcing relations, is also definable. Since $\Tr$ is definable in $V_\kappa[H]$ and fulfills the Tarskian recursion for being a truth predicate for $L_\kappa$, there must be some condition $q\in\Q$ forcing with respect to $\Q$ over $L_\kappa$ that a certain definition $\psi(x,\vec y)$ defines a truth predicate for $L_\kappa$. Since $\Tr$ is the only class that fulfills the property of being an $L_\kappa$-truth predicate---it is unique in all forcing extensions of $V_\kappa$ and indeed of $V$---it follows that all conditions $r\leq q$ must force exactly the same information about instances of $\psi(x,\vec y)$, since this information must accord with $\Tr$. Thus, $\Tr(\varphi,\vec a)$ holds just in case there is in $V_\kappa$ some $r\leq q$ with $r\forces_\Q\psi(\varphi,\check{\vec a})$. Since we assumed that the $\Q$ order and the instances of the forcing relation $\forces_\Q$ are definable classes in $V_\kappa$, it would follow that $\Tr$ is definable in $V_\kappa$, contrary to Tarski's theorem.

So the model $V_\kappa[G]$ in the forcing extension $V[G]$ is an instance of what is desired, a transitive model of \ZFC\ that does not arise as a class-forcing extension of its \HOD\ by any definable class forcing notion with a definable forcing relation, because it admits a definable truth predicate for its \HOD. It follows by a \Lowenheim-Skolem argument that there is also a countable transitive model instance of the phenomenon, and this is a $\Sigma^1_2$ assertion in $V[G]$, which must therefore also be true in the original model $V$. So there is such a transitive model of \ZFC\ in $V$ that is not a class-forcing extension of its \HOD, as desired.
\end{proof}

Let us now explain how to improve the argument by weakening the large cardinal hypothesis. The first thing to notice is that we don't really need the inaccessible cardinal, if we fall back merely on the assumption in \GBC\ that there is a first-order truth predicate. The observations in section~\ref{Section.Background} show that this assumption is strictly weaker in consistency strength (if consistent) than \ZFC\ plus an inaccessible cardinal, since it is provable in \KM\ and even in $\GBC+\ETR_\omega$, a slight strengthening of \GBC. With this weaker assumption, however, we no longer get a transitive model with the desired property, although if we had assumed a transitive model with a truth-predicate, a hypothesis that still remains weaker in consistency strength than an inaccessible cardinal, then the final model would also be transitive.

\begin{theorem}\label{Theorem.Main-theorem-from-truth-predicate}
 If there is a model of \GBC\ having a truth-predicate for first-order truth, then there is a model of \ZFC\ which is not a class-forcing extension of its \HOD\ by any class forcing notion with definable forcing relations in that \HOD.
\end{theorem}

\begin{proof}
Let's simply work inside the model, that is, in \GBC\ under the assumption that there is a truth predicate class $\Tr$, a class (necessarily not first-order definable) obeying the Tarski recursion. Using $\Tr$ we may construct a truth predicate $\Tr^L$ for $L$, and this truth predicate is strongly amenable to $L$, in the sense that if we move to $L$ and take the classes definable in the structure $\<L,\in,\Tr^L>$, we will retain \GBC. To see this, notice first that by the reflection theorem, there are ordinals $\theta$ for which $\<L_\theta,\in,\Tr^L\intersect L_\theta>$ is $\Sigma_1$-elementary in $\<L,\in,\Tr^L>$, and this implies that $\Tr^L\intersect L_\theta$ is the (unique) truth predicate for $\<L_\theta,\in>$, which is an element of $L$. Since by reflection sets defined in $\<L,\in,\Tr^L>$ are also definable in some initial segment $\<L_\theta,\in,\Tr^L\intersect L_\theta>$, it follows that they are all in $L$, and so $\<L,\in,\Tr^L>$ with its definable classes forms a model of \GBC. In short, we may assume $V=L$ without loss.

Next, as in theorem~\ref{Theorem.Main-theorem-using-inaccessible}, let $T\of\Ord$ be a class of ordinals that codes $\Tr^L$ in some canonical way, and let $\P=\prod_{\alpha\in T}\Add(\delta_\alpha,1)$ be the Easton-support product forcing that adds a Cohen subset to $\delta_\alpha=\aleph_{\omega(\alpha+1)+1}$ for all and only the $\alpha\in T$. If $G\of\P$ is $L$-generic, we may form the extension $L[G]$, which satisfies \ZFC. As before, the homogeneity argument will establish $\HOD^{V[G]}=V=L$. And the class $T$ will be definable in $L[G]$ as before, but there can be no class forcing notion definable in $L$ with definable forcing relations in $L$, because from such a forcing relation we would as before be able to define $T$ and therefore $\Tr^L$ in $L$, contrary to Tarski's theorem.
\end{proof}

Ultimately, this proof relies on the fact that $\Tr^L$ is not explicitly first-order definable in $L_\kappa$, although it is nevertheless implicitly first-order definable in $\<L_\kappa,\in>$, in the sense of~\cite{HamkinsLeahy2016:AlgebraicityAndImplicitDefinabilityInSetTheory} and~\cite{GroszekHamkins:The-implicitly-constructible-universe}, meaning that $\Tr^L$ is the unique class even in the extension satisfying a particular first-order expressible property in the language with a class predicate. Namely, $\Tr^L$ is the unique class satisfying the (first-order expressible) property of what it means to be a truth predicate, the Tarski recursion, for the ground model, and furthermore this uniqueness continues to hold in the extension.

Let us now finally prove the main theorem under the optimal hypothesis.

\begin{theorem}\label{Theorem.Main-theorem-optimal}
  If \ZFC\ is consistent, then there is a model of \ZFC\ that is not a class-forcing extension of its \HOD\ by any class forcing notion with definable forcing relations in that \HOD.
\end{theorem}

\begin{proof}
Assume that \ZFC\ is consistent. It follows that there is a countable $\omega$-nonstandard model of set theory $\<M,\in^M,\Tr^*>$ having a partial truth predicate $\Tr^*$ which is defined on the $\Sigma_n$-assertions of the model for some nonstandard $n$, obeying the Tarski recursion on its domain, and where the model satisfies \ZFC\ in the language with the predicate $\Tr^*$. There is such a model because the theory asserting $\ZFC$ in the expanded language plus the assertion that $\Tr^*$ is a partial truth class defined on the $\Sigma_1$ assertions, the $\Sigma_2$-assertions and so on, is finitely consistent and so it has a model. (But actually, a more general fact is that any countable computable saturated model $\<M,\in^M>$ can be expanded to include such a partial truth predicate, and furthermore this characterizes the computably saturated models amongst countable models.)

Even though $\Tr^*$ is only a partial truth predicate, nevertheless Tarski's theorem on the nondefinability of truth still goes through to show that $\Tr^*$ is not definable from parameters in $\<M,\in^M>$, since $\Tr^*$ must agree with the actual satisfaction relation on standard-finite formulas. (See~\cite{HamkinsYang:SatisfactionIsNotAbsolute} for examples of the bizarre features that nonstandard truth predicates can exhibit.)

We now simply undertake the argument of theorem~\ref{Theorem.Main-theorem-from-truth-predicate} using the predicate $\Tr^*$ instead of a full truth predicate. Namely, we may assume $V=L$ in $M$ without loss, and then let $T\of\Ord$ be a class of ordinals coding $\Tr^*$, and then force with $\P$ so as to code $T$ by the class of cardinals for which there is an $L$-generic Cohen subset. Again we will have $\HOD^{M[G]}=M$ by homogeneity, and so $T$ and hence $\Tr^*$ will be definable in the forcing extension $M[G]$. But the standard-finite instances of $\Tr^*$ are uniquely determined by actual truth in $\<M,\in^M>$, and so if the forcing relation is definable, we will be able to define a partial satisfaction class in $\<M,\in^M>$, contrary to Tarski's theorem. So $M[G]$ cannot be realized as a class-forcing extension of its \HOD\ by any class forcing notion definable in its \HOD\ with definable forcing relations there.
\end{proof}

\section{What the argument does not show}

A key idea of the main theorem is that we have a model of set theory and a class $T \subset \Ord$ that was first-order definable in our larger model, but was not first-order definable in \HOD.

We should like briefly to explain that those features alone are not sufficient to conclude that the universe is not a class-forcing extension of \HOD. For example, we might have started in $L$, and then forced to add a generic Cohen subclass $A\of\Ord$, using set-sized conditions, that is, using the partial order $2^{<\Ord}$. It is not difficult to see that $A$ is not definable in $L$. Next, we could force over the \GBC\ model $L[A]$ to code $A$ into the class of ordinals $\alpha$ for which there is an $L$-generic Cohen subset of $\delta_\alpha$, as before. In the resulting forcing extension $L[A][G]$, the class $A$ has become first-order definable. As a \ZFC\ model, it is the same as $\<L[G],\in>$, since the forcing to add $A$ added no sets. Since the coding forcing is weakly homogeneous, it follows as before that $\HOD^{L[G]}=L$. So this is a case where we had a nondefinable class in $L[A]$, which became definable in the forcing extension $L[G]$. But nevertheless, the model $L[G]$ is a class-forcing extension of $L$, since one can combine the two forcing notions into one class forcing notion, whose conditions specify an initial segment of $A$ and a condition in the corresponding forcing to begin coding that initial segment.

The conclusion is that one may not deduce that $V$ is not a class-forcing extension of \HOD\ merely by having a definable class of ordinals in $V$ that is not definable in \HOD. What we had used about the truth predicate $\Tr$ was not only that it was not definable, but that it would remain unique in all extensions of our original ground model. This is how we knew that all conditions in the forcing must agree on the particular details of the content of $\Tr$, and so from the definability of the forcing relation we get $\Tr$ definable in the ground model, contrary to Tarski's theorem. So what was used about $\Tr$ was that it was not only implicitly definable in the extension, but necessarily implicitly definable in any possible extension. 

We could have argued alternatively that if a partial truth predicate for the ground model becomes definable in a forcing extension whose forcing relations are definable, then since the truth values of $\Sigma_k$ formulas uniquely determines the truth values of $\Sigma_{k+1}$ formulas, it follows by induction on $k$ that all conditions must force the same truth values for all $\Sigma_k$ formulas for all $k$, since there can be no least counterexample to this. Thus, the truth predicate would be definable already in the ground model, contrary to Tarski's theorem.

\section{Friedman's theorem}\label{Section.Friedmans-theorem}

Let us now turn to Friedman's theorem~\cite{Friedman2012:TheStableCore}, which shows a sense in which one can attain a positive answer to the main question in the introduction.

\begin{theorem}[Friedman~\cite{Friedman2012:TheStableCore}]\label{Theorem.Friedmans-theorem}
There is a class $A$, definable in $V$ and amenable to \HOD\ in the sense that $\<\HOD,\in,A>$ satisfies \ZFC\ in the expanded language, such that $V$ is a class-forcing extension of $\<\HOD,\in,A>$ by a forcing notion $\P$ that is definable in $\<\HOD,\in,A>$ and which has a definable forcing relation there.
\end{theorem}

In other words, there is a class filter $G\of\P$ meeting every dense class definable in $\<\HOD,\in,A>$ such that $V=\HOD[G]$, meaning that every set $a\in V$ is the value of a name $a=\val(\dot a,G)$ for some name $\dot a\in\HOD$.

This theorem can be seen as a positive answer to the main question, since it provides a way to amalgamate all the various set-forcing extensions $\HOD[G]$ provided by \Vopenka's theorem into one grand class forcing notion. The downside of the theorem is that the forcing notion $\P$ and its forcing relation $\forces_\P$ are not necessarily available as classes in \HOD, if one should consider \HOD\ purely as a \ZFC\ model. So this isn't the usual kind of case of class forcing that one imagines in \ZFC, where we have a ground model and a definable class forcing notion, to be used to make a forcing extension. Rather, for this case, we have the extension universe $V$ already and we can only define the forcing notion $\P$ in the extension $V$, and then observe that indeed it is possible to amend \HOD\ by this forcing notion in such a way so as to realize $V$ as the extension.

\section{The mantle}\label{Section.The-mantle}

It turns out that our main result also answers the following natural analogue of the main question in the area of set-theoretic geology, namely, the question of whether the set-theoretic universe $V$ must arise by class forcing over the mantle.

The subject of set-theoretic geology, introduced in~\cite{FuchsHamkinsReitz2015:Set-theoreticGeology}, is the study of the collection and structure of the ground models of the universe. A transitive inner model $W\satisfies\ZFC$ is a \emph{ground} of the set-theoretic universe $V$ if there is some forcing notion $\P\in W$ and $W$-generic filter $G\of\P$ for which $V=W[G]$ is the resulting forcing extension. The ground-model definability theorem, due to Laver~\cite{Laver2007:CertainVeryLargeCardinalsNotCreated} and Woodin~\cite{Woodin2004:RecentDevelopmentsOnCH} (Laver~\cite{Laver2007:CertainVeryLargeCardinalsNotCreated} uses a proof due to Hamkins), asserts that every such ground model $W$ is definable from parameters in $V$. The ground-model enumeration theorem~\cite[theorem~12]{FuchsHamkinsReitz2015:Set-theoreticGeology} asserts that there is a definable class $W\of V\times V$ whose sections $W_r=\set{x\mid (r,x)\in W}$ are all ground models of $V$ and every ground model of $V$ is such a $W_r$. In this way, quantifying over the ground models becomes a first-order quantifier, by quantifying over the indices $r$ that give rise to the various ground models $W_r$.

The \emph{mantle} $\Mantle$ is simply the intersection of all ground models $\Mantle=\Intersect_r W_r$, which is a definable class in light of the ground-model enumeration theorem. Results in~\cite{FuchsHamkinsReitz2015:Set-theoreticGeology}, when combined with Usuba~\cite{Usuba2017:The-downward-directed-grounds-hypothesis-and-very-large-cardinals}, show that the mantle is an inner model of \ZFC, and indeed, it is the largest forcing-invariant definable class.

Since the mantle arises through the process of stripping away the various forcing extensions that had given rise to the universe $V$, it is natural to inquire:

\begin{question}
 Is the set-theoretic universe $V$ necessarily a class-forcing extension of the mantle?
\end{question}

This is an analogue of the main question in the context of set-theoretic geology, and we intend again that the question is asking whether there must be a class forcing notion $\P$ definable in the mantle and with definable forcing relations there, such that $V$ arises as a forcing extension $V=\Mantle[G]$ for some $\Mantle$-generic $G\of\P$.

Once again, the answer is negative.

\begin{theorem}
 If \ZFC\ is consistent, then there is a model of \ZFC\ that does not arise as a class-forcing extension of its mantle $\Mantle$ by any class forcing notion with definable forcing relations in $\Mantle$.
\end{theorem}

\begin{proof}
We may use the same arguments as in theorems~\ref{Theorem.Main-theorem-using-inaccessible},~\ref{Theorem.Main-theorem-from-truth-predicate} and~\ref{Theorem.Main-theorem-optimal}, because in those cases, the mantle of the resulting final model will be the same as its \HOD. The point is that since the forcing there was an Easton-support product forcing over $L$, one may strip off any particular factor or set of factors of the forcing and thereby achieve a ground. That is, the final model $L[G]$ is obtained by set-forcing over the tail-forcing ground $L[G^\kappa]$, where $G^\kappa$ performs the forcing only at factors $\kappa$ and above. The intersection of these tail-forcing grounds is simply $L$, and so the mantle of the model is $L$, which as we argued before, is the same as the \HOD\ of the model. Thus, the previous arguments show in those cases that the universe is not obtained by class-forcing over the mantle.
\end{proof}

\section{The intermediate model theorem}\label{Section.Intermediate-models}

Let us now explain how our main theorem provides some simple counterexamples to the class-forcing analogue of the intermediate-model theorem, the assertion that any \ZFC\ model intermediate between a ground and a forcing extension is also a forcing extension of that ground. In the case of set-forcing, the basic fact is familiar:

\begin{theorem}[{\cite[corollary~15.43]{Jech:SetTheory3rdEdition}}]\label{Theorem.Intermediate-model-theorem}
  If $M\of W\of M[G]$ are \ZFC\ models, where $M[G]$ is a forcing extension of $M$ by set-sized forcing in $M$ and $W$ is an inner model of $M[G]$, then $W$ is also a forcing extension of $M$ and a ground of $M[G]$.
\end{theorem}

Furthermore, if $G\of\B$ is $M$-generic over a complete Boolean algebra $\B\in M$, then there is a complete subalgebra $\B_0\of\B$ in $M$ such that $W=M[G_0]$ is the forcing extension via $G_0=G\intersect\B_0$, and consequently also $M[G]$ is a forcing extension of $W$ by the quotient forcing.

It is natural to inquire whether the property holds for class-forcing extensions.

\begin{question}
 Does the intermediate-model property hold for class forcing? In other words, if $V\of V[G]$ is a class-forcing extension of the universe $V$ and $W$ is a transitive \ZFC\ model with $V\of W\of V[G]$, then must $W$ be a class-forcing extension of $V$? Must $W$ be a ground of $V[G]$?
\end{question}

The answer to both of these questions is no, if one insists that the forcing notions be definable classes in $V$ with definable forcing relations there. Let us begin with the simplest kind of counterexample.

\begin{theorem}\label{Theorem.Counterexample-to-intermediate-model-property}
  If \ZFC\ is consistent, then there are models of \ZFC\ set theory $V\of W\of V[G]$, where $V[G]$ is a class-forcing extension of $V$ and $W$ is a transitive inner model of $V[G]$, but $W$ is not a forcing extension of $V$ by any class forcing notion with definable forcing relations in $V$.
\end{theorem}

\begin{proof}
Let us start with any \GBC\ model $V$ having a partial satisfaction class $\Tr^*$, and by further forcing if necessary, let us assume that the \GCH\ holds. (This situation is equiconsistent with \ZFC, as we argued in theorem~\ref{Theorem.Main-theorem-optimal}.) Let $T\of\Ord$ be a class of ordinals coding $\Tr^*$ and let $W=V[G_0]$ be the forcing extension that forces with the Easton-support product forcing so as to code $T$ into the \GCH\ pattern on the regular cardinals. Let $V[G]$ be the forcing extension that performs the \GCH\ failure forcing at all the other regular cardinals as well, effectively erasing the coding.

Now, we throw away all the classes and consider $V\of W\of V[G]$ as \ZFC\ models. Note that $V[G]$ is a class-forcing extension of $V$ by the Easton-support product forcing $\prod_\gamma\Add(\gamma,\gamma^{++})$ forcing the failure of the \GCH\ at every regular cardinal $\gamma$. This forcing is definable and tame in $V$ and has definable forcing relations there. But $W$ cannot arise as such a forcing extension of $V$, since $T$ and hence $\Tr^*$ are definable in $W$, and from this it implies by (internal) induction on formulas that all conditions must force the same outcomes for that truth predicate, and so $\Tr^*$ would be definable in the ground model $V$, contrary to Tarski's theorem.
\end{proof}

Failures of the intermediate-model property for class forcing were already provided by S.~Friedman~\cite{Friedman1999:Strict-genericity}, where he used $0^\sharp$ to construct counterexamples. Our example here, however, appears to be considerably simpler and we have omitted the use of $0^\sharp$. In addition, our method works over models of arbitrary consistent first-order extensions of \GBC, because every consistent extension of \GBC\ has models with a partial satisfaction class $\Tr^*$, and this enables the construction of theorem~\ref{Theorem.Counterexample-to-intermediate-model-property}. In particular, there is no way to avoid the counterexamples by making large cardinal assumptions or other first-order assumptions about the nature of the ground model $V$. The forcing in our example is also comparatively mild, a progressively closed Easton-support product (and see theorem~\ref{Theorem.Counterexample-intermediate-model-property-CBA}, where we realize it via an $\Ord$-c.c.~complete Boolean algebra class). In theorem~\ref{Theorem.Intermediate-double-counterexample}, we provide a counterexample violating the class-forcing intermediate model property simultaneously in both directions.

One might hope to avoid the counterexamples by making stronger assumptions on the nature of the forcing giving rise to the extension $V[G]$. For example, since the standard proof of theorem~\ref{Theorem.Intermediate-model-theorem} in the set-forcing case makes fundamental use of the Boolean algebra for the forcing, rather than just the forcing notion as a partial order, one might hope that the intermediate-model property would hold for instances of class forcing where the forcing notion is a complete Boolean algebra.

\begin{question}
 If $V\of W\of V[G]$, where $G\of\B$ is $V$-generic for a definable class forcing notion $\B$, which is an $\Ord$-c.c.~complete class Boolean algebra in $V$ and $W$ is a $\ZFC$-inner model of $V[G]$, then must $W$ be a forcing extension of $V$? Must $W$ be a ground of $V[G]$?
\end{question}

The answer is negative. We are grateful to Victoria Gitman for discussions and observations concerning this question and for her suggestion to consider ``$\Ord$ is Mahlo.''

\begin{theorem}\label{Theorem.Counterexample-intermediate-model-property-CBA}
 If $\ZFC+``\Ord$ is Mahlo'' is consistent, then there are models of \ZFC\ set theory $V\of W\of V[G]$, where $G\of\B$ is $V$-generic by a tame definable $\Ord$-c.c.~complete Boolean algebra class $\B$, but $W$ is not a forcing extension of $V$ by any definable class forcing notion with definable forcing relations in $V$.
\end{theorem}

\begin{proof}
The point is that if $\Ord$ is Mahlo, then the Easton-support product forcing used in the proof of theorem~\ref{Theorem.Counterexample-to-intermediate-model-property} is $\Ord$-c.c., by the class analogue of the usual proof that the Easton-support product of small forcing up to a Mahlo cardinal $\kappa$ is $\kappa$-c.c. Since the forcing is tame, it has a definable forcing relation, and therefore by~\cite[theorem~14]{GitmanHamkinsHolySchlichtWilliams:The-exact-strength-of-the-class-forcing-theorem} it has a definable set-complete Boolean completion. Since the forcing is $\Ord$-c.c., set-complete means fully complete. So when $\Ord$ is Mahlo, the class forcing $\B$ used in theorem~\ref{Theorem.Counterexample-to-intermediate-model-property} can be taken to be a complete Boolean algebra. But still the argument of that theorem shows that the intermediate model $W$ cannot be a class-forcing extension of $V$ by any class forcing notion with definable forcing relations in $V$.
\end{proof}

So far, the counterexamples provided by theorems~\ref{Theorem.Counterexample-to-intermediate-model-property} and~\ref{Theorem.Counterexample-intermediate-model-property-CBA} violate only one part of the intermediate-model property. Namely, in the set-forcing case, when $V\of W\of V[G]$ is an intermediate model for set-sized forcing, then the intermediate model $W$ is both a forcing extension of $V$ and a ground of $V[G]$. In our counterexamples above, we only showed that $W$ is not a forcing extension of $V$. And in those cases, $W$ was in fact a ground model of $V[G]$, since the rest of the forcing from $W$ to $V[G]$ was in fact definable in $W$.

Similar arguments, however, provide a dual kind of counterexample, with $V\of W\of V[G]$, but where $W$ is not a ground of $V[G]$ by any class forcing notion definable in $W$ and with definable forcing relations there. Namely, in order to achieve this, let  $V=L[G_T]$, where the class $T$ coding $\Tr^L$ is coded. Next, do the erasing forcing to $W=L[G]$, which arises as an ordinary Easton-support product of Cohen forcing. So the class $T$ will no longer be definable in $W$. Finally, form $V[G]=W[H_T]$ by coding $T$ again, on a disjoint noninterfering class of coding points. The final model $V[G]$ is a class-forcing extension of $V=L[G_T]$, since $T$ is definable in $V$. But it is not a class-forcing extension of $W$, since $T$ is not definable there.

In order to provide a full counterexample to the class-forcing intermediate model property $V\of W\of V[G]$, where simultaneously $W$ is not a class-forcing extension of $V$ and also not a ground of $V[G]$, we have to push the methods a bit harder.

\begin{theorem}\label{Theorem.Intermediate-double-counterexample}
 If \ZFC\ is consistent, then there are models of \ZFC\ set theory 
  $$V\of W\of V[G],$$
 where $V[G]$ is a class-forcing extension of $V$ and $W$ is an inner model of $V[G]$, transitive with respect to $V[G]$, but $W$ is neither a forcing extension of $V$ nor a ground of $V[G]$ by any class forcing notion with definable forcing relations in these respective ground models. If\ $\ZFC+\Ord$ is Mahlo is consistent, then one can find such instances where also $V\of V[G]$ is a generic extension by $G\of\B$, where $\B$ is an $\Ord$-c.c.~complete Boolean algebra class definable in $V$.
\end{theorem}

\begin{proof}
If \ZFC\ is consistent, then so is \GBC\ and furthermore, we have mentioned that there is a model of \GBC\ with a partial satisfaction class, which is a truth predicate defined on the $\Sigma_n$ formulas for some nonstandard $n$ and obeying the Tarskian recursion up to that level. Note that any truth predicate defined up to $\Sigma_n$ can be extended uniquely to $\Sigma_{n+1}$ by unwrapping one additional level of the Tarski recursion.

Since truth predicates are not definable, it follows from classical results of Krajewski~\cite{Krajewski1974:MutuallyInconsistentSatisfactionClasses, Krajewski1976:Non-standardSatisfactionClasses} (see also~\cite{HamkinsYang:SatisfactionIsNotAbsolute}) that there is a model $M\satisfies\ZFC+V=L$ with two different partial satisfaction classes $\Tr_0$ and $\Tr_1$, defined on all $\Sigma_n$ formulas for some $n$, which disagree on truth at some level and which are both \ZFC-amenable to $M$. The model $M$ is necessarily $\omega$-nonstandard, since an induction on formulas in the meta-theory shows that any two truth predicates must agree on the truth of all standard formulas; so the formulas on which they disagree must be nonstandard. A similar argument shows that there can be no least $n$ where the disagreement between $\Tr_0$ and $\Tr_1$ arises, and so the classes $\Tr_0$ and $\Tr_1$ cannot be jointly \ZFC-amenable to $M$, since in $\<M,\in,\Tr_0,\Tr_1>$ we would thereby violate induction in the natural numbers. But considered separately, the structures $\<M,\in,\Tr_0>$ and $\<M,\in,\Tr_1>$ satisfy \ZFC\ in the expanded language and constitute \GBC\ models with their definable classes in the expanded language.

Let's now construct the desired counterexample. Fix three disjoint scattered non-interfering proper classes of coding points, $\gamma_\alpha$, $\delta_\alpha$, $\eta_\alpha$, which are all regular cardinals in $M$ and where the enumeration definitions are absolute. 

We begin with the structure $\<M,\in,\Tr_0>$. Using the first class of coding points $\gamma_\alpha$, let $M[G_0]$ be a forcing extension of $\<M,\in,\Tr_0>$ in which the class $\Tr_0$ becomes definable, by means of the Easton-support product of the forcing that adds a generic Cohen subset to $\gamma_\alpha$ whenever $\alpha$ is the code of an element of $\Tr_0$. 
After this, let us force to `erase' this coding, by adding Cohen subsets to all the remaining $\gamma_\alpha$, so that the combined forcing simply adds Cohen subsets to every $\gamma_\alpha$. So $M\of M[G_0]\of M[G]$, where $G$ is $M$-generic for the Easton-support product forcing that adds a subset to every $\gamma_\alpha$. We may assume that $G$ is not only $M$-generic but also $\<M,\in,\Tr_0>$-generic and $\<M,\in,\Tr_1>$-generic, if we simply construct $G$ first and then construct $G_0$ from it by restricting to the coordinates coding elements of $\Tr_0$.

Next, viewing $M[G]$ as a \GBC\ model having arisen as a forcing extension of $\<M,\in,\Tr_1>$, we force over $M[G]$ to $M[G][H_0]$, where $\Tr_1$ becomes coded by $H_0$ on the coding points $\delta_\alpha$. And then we perform the erasing forcing, adding Cohen sets to all the remaining $\delta_\alpha$, to form the extension $M[G][H]$. The forcing to add $H$ is the Easton-support product of the forcing to add a Cohen subset to every $\delta_\alpha$, and this forcing is definable in $M$ with a definable forcing relation there. We may assume that $G\times H$ is not only $M$-generic, but also $\<M,\in,\Tr_0>$ and $\<M,\in,\Tr_1>$-generic, by selecting $G$ and $H$ first and then restricting to $G_0$ and $H_0$.

Finally, viewing $M[G][H]$ as a GBC forcing extension of $\<M,\in,\Tr_0>$, we force to code $\Tr_0$ once again on the coding points $\eta_\alpha$ with filter $K_0$, forming the extension $M[G][H][K_0]$. 

The full chain of forcing extensions is
 $$M\of M[G_0]\of M[G] \of M[G][H_0]\of M[G][H]\of M[G][H][K_0],$$
and we view them now merely as \ZFC\ models, throwing all the classes away. Consider specifically the three models: 
 $$M[G_0]\quad\of\quad M[G][H_0]\quad\of\quad M[G][H][K_0].$$
Notice that since $\Tr_0$ is definable in $M[G_0]$, we may easily define in $M[G_0]$ the forcing to the final model $M[G][H][K_0]$, since this is just the erasing forcing on the $\gamma_\alpha$, the full forcing on the $\delta_\alpha$, and then the forcing to code $\Tr_0$ again on the $\eta_\alpha$. So we have a ground model $M[G_0]$ and a forcing extension $M[G][H][K_0]$ by definable class forcing in that ground model with a definable forcing relation there. 

But consider the intermediate model $M[G][H_0]$. The truth predicate $\Tr_1$ is definable there, being coded by $H_0$ on the cardinals $\delta_\alpha$. If $M[G][H_0]$ could be realized as a class-forcing extension of $M[G_0]$, by some class forcing notion with definable forcing relations, then there would be some condition $p$ forcing that the existence of $L$-generic Cohen sets on the cardinals $\delta_\alpha$ coded a partial satisfaction class for $M$ on the $\Sigma_n$ formulas, the same (nonstandard) $n$ as in $\Tr_1$. In $M[G_0]$, we may consider which truths were forced by which conditions in this forcing. By (internal) induction on formulas, it follows that all conditions below $p$ must force the same truth values for any particular formula in the truth predicate coded on the $\delta_\alpha$, since this is true for the atomic formulas and inductively if all conditions force the same outcome for the subformulas of a formula, then they also force the same outcome for the formula itself by the Tarski recursion. Therefore, $\Tr_1$ must be definable in $M[G_0]$, where $\Tr_0$ is also definable. But it is easy to see by induction on formulas that all partial truth predicates must agree on their common domain, contrary to our assumption via Krajewsky that $\Tr_0$ and $\Tr_1$ disagree. So $M[G][H_0]$ is not a class-forcing extension of $M[G_0]$ by any class forcing notion with definable forcing relations in $M[G_0]$.

With exactly similar reasoning, we claim that the final model $M[G][H][K_0]$ is not a class-forcing extension of the intermediate model $M[G][H_0]$, because $\Tr_1$ is definable in the intermediate model and $\Tr_0$ is definable in the final model. So the intermediate model $M[G][H_0]$ is neither a forcing extension of $M[G_0]$ nor a ground of $M[G][H][K_0]$, even though it is between these two models, which form a class-forcing extension. 

If we had $\Ord$ is Mahlo in $M$, then just as in theorem~\ref{Theorem.Counterexample-intermediate-model-property-CBA} the forcing to go from $M[G_0]$ to $M[G][H][K_0]$ would be $\Ord$-c.c., and therefore could be realized as a class-forcing extension by a complete Boolean algebra class.
\end{proof}

If we had assumed that \GBC\ was relatively consistent with the existence of a truth predicate, a mild extra assumption that is weaker than $\GBC+\ETR_\omega$, which lies strictly below Kelley-Morse set theory, then we could have avoided the fuss in the argument about partial truth predicates and worked instead with distinct full truth predicates. 

We do not yet know how to arrange counterexamples of the kind mentioned in theorem~\ref{Theorem.Intermediate-double-counterexample} with transitive models. 

\printbibliography

\end{document}